\documentclass[12pt]{article}

\usepackage{amsmath,amssymb,amsfonts, amsthm, amscd} 
\usepackage{geometry}  
\usepackage{latexsym}  
\usepackage{eucal} 
\usepackage{setspace}
\def\B{\mathcal{ B}}
\def\C{\mathbb{ C}}

\def\epsilon{\varepsilon}

\newtheorem{thm} {Theorem}[section]
\newtheorem{cry}[thm]{Corollary}
\newtheorem{lemma}[thm]{Lemma}
\newtheorem{prop}[thm]{Proposition}
\numberwithin{equation}{section}

\begin{document}

\title{SOME CLOSED RANGE INTEGRAL OPERATORS ON SPACES OF ANALYTIC FUNCTIONS}
\author{Austin Anderson\thanks{The author was supported by NSF-DGE-0841223.}}
\date{March 9, 2010}
\maketitle

\noindent \textbf{Abstract:} Our main result is a characterization of $g$ for which the operator $S_g(f)(z) = \int_0^z f'(w)g(w)\, dw$ is bounded below on the Bloch space. We point out analogous results for the Hardy space $H^2$ and the Bergman spaces $A^p$ for $1 \leq p < \infty$. We also show the companion operator $T_g(f)(z) = \int_0^z f(w)g'(w) \, dw$ is never bounded below on $H^2$, Bloch, nor BMOA, but may be bounded below on $A^p$.
\vspace{0.1in}

\noindent \textbf{Keywords:} Volterra operator, Cesaro operator, integral operator, bounded below, closed range, Bloch, Hardy, Bergman, BMOA, multiplication operator

\section{Introduction}

We examine operators on Banach spaces of analytic functions on the unit disk in the complex plane. The operator $T_g,$ with symbol $g(z)$ an analytic function on the disk, is defined by
$$T_gf(z) = \int_0^z f(w) g'(w) \, dw.$$ 
$T_g$ is a generalization of the standard integral operator, which is $T_g$ when $g(z) = z$. 
Letting $g(z) = \log(1/(1-z))$ gives the Ces\'aro operator. Discussion of the operator $T_g$ first arose in connection with semigroups of composition operators. (see \cite{SZ} for background) Characterizing the boundedness and compactness of $T_g$ on certain spaces of analytic functions is of recent interest, as seen in \cite{AC}, \cite{AS}, \cite{Do} and \cite{SZ}, and open problems remain. $T_g$ and its companion operator $S_gf(z) = \int_0^z f'(w) g(w) \, dw$ are related to the multiplication operator $M_gf(z) = g(z)f(z)$, since integration by parts gives
\begin{equation}
\label{prod}
M_gf = f(0)g(0) + T_gf + S_gf.
\end{equation}

If any two of $M_g$, $S_g$, and $T_g$ are bounded, then so is the third. But in some situations one operator is bounded while two are unbounded. Boundedness of $T_g$ on the Hardy and Bergman spaces and BMOA is characterized in \cite{AC}, \cite{AS} and \cite{SZ}. The pointwise multipliers of these and many other spaces are well known. See \cite{St} for BMOA.

In this paper we examine the property of being bounded below for $T_g$ and $S_g$ on spaces of analytic functions. We examine aspects of the problems on Hardy and Bergman spaces, the Bloch space, and BMOA. In doing so we must assume the operators are bounded, and we study characterizations of the symbols for which the operators are bounded. Consideration of $M_g$ is useful as well.

\section{Preliminaries}

The notation $f \lesssim g$ will mean there exists a universal constant $C$ such that $f \leq Cg$. $f \approx g$ will mean $f \lesssim g \lesssim f$.

Let $D$ be the unit disk in the complex plane. Let $H(D)$ denote the set of analytic functions on $D$. For $1 \leq p < \infty$, the Hardy space $H^p$ on $D$ is
$$ \left\{f \in H(D) : \| f \|_p = \sup_{0<r<1} \int_0^{2\pi} |f(re^{it})|^p dt < \infty \right\}.$$
The space of bounded analytic functions on $D$ is
$$ H^\infty = \{f \in H(D) : \| f \|_\infty = \sup_{z \in D} |f(z)| < \infty\}.$$
We define weighted Bergman spaces, for $\alpha > -1$,
$$ A^p_\alpha = \left\{f \in H(D) : \| f \|_{A^p_\alpha} = \int_D |f(z)|^p(1-|z|^2)^\alpha dA(z) < \infty \right\},$$
where $dA(z)$ refers to Lebesgue area measure on $D$.

The Bloch space is
$$ \B = \left\{f \in H(D) : \| f \|_\B = \sup_{z \in D} |f'(z)|(1-|z|^2) < \infty\right\}.$$
Note that $\| \|_\B$ is a semi-norm. The true norm accounts for functions differing by an additive constant.

A complex measure $\mu$ on $D$ is called a (Hardy space) Carleson measure if there exists $C > 0$ such that $\mu(S(I)) \leq C|I|$ for all arcs $I \subseteq \partial D$, where $S(I) = \{re^{i\theta} : 1-|I| < r < 1, e^{i\theta} \in I\}$ is the Carleson rectangle associated with $I$, and $|I|$ is the length of $I$. The smallest such $C$ is called the Carleson constant for the measure $\mu$. Define, for $f \in H(D)$, $d\mu_f(z) = |f'(z)|^2(1-|z|^2) dA(z)$. The space of analytic functions of bounded mean oscillation, BMOA, is the set of $f$ for which $\mu_f$ is Carleson. The BMOA norm $\|f\|_*$ is comparable to the square root of the Carleson constant for $\mu_f$. The space of analytic functions of vanishing mean oscillation, VMOA, is the set of $f$ for which
$$ \lim_{|I| \to 0} \frac{\mu_f(S(I))}{|I|} = 0.$$
Zhu \cite{Z} is a good reference for background on all these spaces.

$H^\infty$ is a subspace of BMOA, which in turn is a subspace of $H^2$. $H^\infty$ is also a subspace of $\B$. The next lemma will be useful later when studying $T_g$.

\begin{lemma}
\label{zn}
Let $f_n(z) = z^n, n = 1, 2, \dots$. $\|f_n\|_X \approx 1$ for all $n$ and $X = H^2, \B, BMOA$.
\end{lemma}
\begin{proof}
It is well-known that $\|f_n\|_{H^2} = 1$ for all $n$. Checking the Bloch norm with a calculation, we get $\|f_n\|_\B \approx \sup_{0<r<1} nr^{n-1}(1-r) = (1 - 1/n)^{n-1} \to 1/e$ as $n \to \infty$. Finally, $1 = \|f_n\|_\infty \lesssim \|f_n\|_{BMOA} \lesssim \|f_n\|_{H^2}$. (see \cite{Z})
\end{proof}

When studying $T_g$ and $S_g$, it is useful to be able to compare the norm of a function to the norm of its derivative. For $p \geq 1$, $\alpha > -1$, the differentiation operator and its inverse, the indefinite integral, are isometries between $A^p_{\alpha}/\C$ and $A^p_{\alpha+p}$, i.e.,
\begin{equation}
\label{isom}
\| f \|_{A^p_\alpha} \approx |f(0)| + \| f' \|_{A^p_{\alpha+p}}.
\end{equation}
(see \cite[Theorem 4.28]{Z}) Making the natural definition $A^2_{-1} = H^2$, the identity holds for $p=2, \alpha=-1$ as well. This is the well-known Littlewood-Paley identity, $\|f\|_{H^2} \approx |f(0)| + \int_D |f'(z)|^2(1-|z|^2)dA(z)$.

On all the spaces mentioned, point evaluation is a bounded linear functional. The norm of point evaluation at $z$ in $A^p_\alpha$ is comparable to $1/(1-|z|)^{(2+\alpha)/p}$. (see \cite[Theorem 4.14]{Z}) In $\B$ and BMOA, the norm of point evaluation is comparable to $\log(2/(1-|z|))$. The following theorem is a generalization of a result on multipliers of Banach spaces in which point evaluation is bounded. See, for example, \cite[Lemma 11]{DRS}.

\begin{thm}
\label{nec}
Let $X, Y$ be Banach spaces of analytic functions, and let $\lambda^0_z$ and $\lambda^1_z$ be linear functionals on $X$ and $Y$ defined by $\lambda^0_z f = f(z)$ and $\lambda^1_z f = f'(z)$. Suppose $\lambda^0_z$ and $\lambda^1_z$ are bounded.
\begin{itemize}
\item[(a)] Suppose $S_g$ maps $X$ boundedly into $Y$. Then 
$$|g(z)| \leq \, \|  S_g  \| \, \frac{\|  \lambda^1_z  \|_Y }{\|  \lambda^1_z  \|_X}.$$
\item[(b)] Suppose $T_g$ maps $X$ boundedly into $Y$. Then 
$$|g'(z)| \leq \, \|  T_g  \| \, \frac{\|  \lambda^1_z  \|_Y }{\|  \lambda^0_z  \|_X}.$$
\end{itemize}
\end{thm}

\begin{proof}
Note that, for $f \in X$,
$$ |f'(z)||g(z)| = |\lambda^1_z S_g(f)| \leq \| \lambda^1_z \|_Y \| S_g \| \| f \|_X $$
Since $\sup_{\|f\|_X = 1} |f'(z)| = \|\lambda^1_z\|_X$, taking the sup over $\|f\|_X = 1$ of both sides gives us
$$ \|\lambda^1_z\|_X |g(z)| \leq \| S_g \| \|\lambda^1_z\|_Y. $$
Hence a). Similarly,
$$ |f(z)||g'(z)| = |\lambda^1_z T_g(f)| \leq \| \lambda^1_z \|_Y \| T_g \| \| f \|_X. $$
Taking the sup over $f$ with norm $1$, we get
$$ \|\lambda^0_z\|_X |g'(z)| \leq \| T_g \| \|\lambda^1_z\|_Y. $$
This completes the proof.
\end{proof}

\begin{cry}
\label{cor_bdd}
If $X$ is a Banach space of analytic functions on which point evaluation of the derivative is a bounded linear functional, and $S_g$ is bounded on $X$, then $g$ is bounded.
\end{cry}

In \cite{CM}, we see a similar result for $M_g$, i.e., boundedness of $M_g$ on a Banach space in which point evaluation is bounded implies $g$ is bounded. On the Hardy and Bergman spaces on the disk, both $M_g$ and $S_g$ are bounded if and only if $g$ is bounded. That this is necessary for $S_g$ is Corollary \ref{cor_bdd}. That it is sufficient follows from integration by parts since $T_g$ and $M_g$ are bounded if $g$ is bounded. (see \cite{AC} and \cite{AS} concerning $T_g$) A similar situation holds for $\B$.

\begin{prop}
\label{sgbloch}
$S_g$ is bounded on $\B$ if and only if $g \in H^\infty$.
\end{prop}
\begin{proof}
It is clear $g \in H^\infty$ implies $S_g$ is bounded, since
$$ \|S_gf\|_\B = \sup_{z \in D} (|f'(z)||g(z)|(1-|z|^2)) \leq \|g\|_{H^\infty} \|f\|_\B $$
The converse follows from Corollary \ref{cor_bdd}.
\end{proof}

\section{The property of being bounded below}

An operator $T$ is said to be bounded below if there exists $C > 0$ such that $\| Tf \| \geq C \| f \|$ for all $f$.

It typically is the case for one-to-one operators on Banach spaces that boundedness below is equivalent to having closed range. The analogue of Theorem \ref{closed_range} for composition operators is found in Cowen and MacCluer \cite{CM}. We include the proof for $T_g$ and $S_g$, essentially the same, for easy reference.

\begin{lemma}
\label{onetoone}
$T_g$ is one-to-one for nonconstant $g$.
\end{lemma}
\begin{proof}
If $T_gf_1 = T_gf_2$, taking derivatives gives $f_1(z)g'(z) = f_2(z)g'(z)$. Thus $f_1(z) = f_2(z)$ except possibly at the (isolated) points where $g$ vanishes. Since $f_1$ and $f_2$ are analytic, $f_1 = f_2$.
\end{proof}

When considering the property of being bounded below for $S_g$, we note that $S_g$ maps any constant function to the $0$ function. Thus, it is only useful to consider spaces of analytic functions modulo the constants.

\begin{thm}
\label{closed_range}
Let $Y$ be a Banach space of analytic functions on the disk. For nonconstant $g$, $T_g$ is bounded below on $Y$ if and only if it has closed range. $S_g$ is bounded below on $Y / \C$ if and only if it has closed range on $Y / \C$.
\end{thm}
\begin{proof}
Assume $T_g$ is bounded below, i.e., there exists $\varepsilon > 0$ such that $\|T_gf\| \geq \varepsilon \|f\|$ for all $f$. Suppose $\{T_g f_n\}$ is a Cauchy sequence in the range of $T_g$. Since $\|f_n - f_m\| \lesssim \|T_g f_n - T_g f_m\|$, $\{f_n\}$ is also a Cauchy sequence. Letting $f = \lim f_n$, we have $T_g f_n \to T_g f$, showing $T_g f_n$ converges in the range of $T_g$. Hence the range is closed.

Conversely, assume $T_g : Y \to Y$ is closed range. Let $\{f_n\}$ be a sequence in $Y$ such that $\|T_g f_n\| \to 0$. $T_g$ is one-to-one by Lemma \ref{onetoone}. Let the closed range of $T_g$ be $X$. $X$ is a Banach space, and we can define the inverse $T_g^{-1} : X \to Y$. Suppose $\{x_n\}$ converges to $x = T_g h$ in $X$, and $T_g^{-1} x_n$ converges to $y$ in $Y$. Applying $T_g$ to $\{T_g^{-1} x_n\}$, this means $x_n$ converges to $T_g y$. Hence $T_g y = T_g h$. Since $T_g$ is one-to-one, $y = h$, and $x = T_g^{-1} y$. By the Closed Graph Theorem, $T_g^{-1}$ is continuous. Thus, $\|f_n\| = \|T_g^{-1}(T_g f_n)\| \to 0$, implying $T_g$ is bounded below.

The same argument holds for $S_g$ as well, but only on spaces modulo constants, since $S_g$ is not one-to-one otherwise.
\end{proof}

We will show that $T_g$ is never bounded below on $H^2$, $\B$, nor BMOA. The sequence $\{z^n\}$ demonstrates the result in each space, since the functions $z^n$ have norm comparable to $1$, independent of $n$. (Lemma \ref{zn})

\begin{thm} 
\label{tgnot}
$T_g$ is never bounded below on $H^2$, $\B$, nor $BMOA$.
\end{thm}

\begin{proof}
Let $f_n(z) = z^n$. For $H^2$,
$$\lim_{n\rightarrow \infty} \|  T_g f_n  \|_2^2 \,  \approx 
\lim_{n\rightarrow \infty} \int_D |z^n|^2 |g'(z)|^2 (1-|z|^2) \, dA(z).$$
We assume $T_g$ is bounded, so $g \in BMOA$ by a result of Aleman and Siskakis. \cite{AS}
Thus $\mu_g$ is a Carleson measure, allowing us to bring the limit inside the integral by the Dominated Convergence Theorem.
$$\lim_{n \to \infty} \|  T_g f_n  \|_2^2 
=  \int_D \lim_{n\rightarrow \infty} |z^n|^2 |g'(z)|^2 (1-|z|^2) \, dA(z)
 =  0.$$
Since $\|  f_n  \|_2 = 1$ for all $n$, $T_g$ is not bounded below.

If $T_g$ is bounded on $\B$, then, by Theorem \ref{nec}, $|g'(z)| (1-|z|) = O(1/\log(1/(1-|z|)))$ as $|z| \to 1$.
$$ \|  T_g f_n  \|_{\B} = \sup_{z \in D} |z^n||g'(z)| (1-|z|) \lesssim \sup_{0\leq r<1}r^n \frac{1}{\log(2/(1-r))}. $$
Given $\varepsilon > 0$, there exists $\delta < 1 $ such that $1/\log(2/(1-r)) < \varepsilon$ for $\delta < r < 1$. For large $n$, $r^n < \varepsilon$ for $0 < r < \delta$. Thus, $\lim_{n \to \infty} \| T_g f_n \|_{\B} = 0$, and Lemma \ref{zn} implies $T_g$ is not bounded below on $\B$.

On BMOA, Siskakis and Zhao proved $T_g$ being bounded implies $g \in VMOA$. \cite{SZ}
$$ \lim_{n \to \infty} \|T_gf_n \|_*^2 \approx \lim_{n \to \infty} \sup_{I} \frac{1}{|I|} \int_{S(I)} |z^n|^2|g'(z)|^2(1-|z|^2) dA(z). $$
Let $I$ be an arc in $\partial D$, and let $\varepsilon > 0$. Since $g \in VMOA$, there exists $\delta > 0$ such that 
$$ \frac{1}{|J|} \int_{S(J)} |g'(z)|^2(1-|z|^2) dA(z) < \varepsilon \text{ whenever } |J| < \delta. $$
If $|I| > \delta$, divide $I$ into $K$ disjoint intervals of length approximately $\delta$, so $I = \cup_{i=1}^K J_i, \delta/2 < |J_i| < \delta$ for all $i$, and $\delta K \approx |I|$.
Let $S_\delta(I) = S(I) - \cup_i S(J_i)$. For large $n$, $(1-\delta/2)^{2n} \leq \varepsilon|I|$, and to estimate the integral over $S_\delta(I)$ we use the fact that $\mu_g$ is a Carleson measure.
$$ \frac{1}{|I|} \int_{S(I)} |z^n|^2|g'(z)|^2(1-|z|^2) dA(z) = \frac{1}{|I|} \int_{S_\delta(I)} |z^n|^2|g'(z)|^2(1-|z|^2) dA(z) $$
$$ + \frac{1}{|I|} \sum_{i=1}^K \int_{S(J_i)} |z^n|^2|g'(z)|^2(1-|z|^2) dA(z) $$
$$ \leq \frac{1}{|I|} (1-\delta/2)^{2n} C\|g\|_*^2 + \frac{1}{|I|} K\delta\varepsilon \lesssim \varepsilon $$
for large $n$. Hence $\lim_{n \to \infty} \|T_gf_n\|_* = 0$ and $T_g$ is not bounded below on BMOA.
\end{proof}

In contrast to Theorem \ref{tgnot}, $T_g$ can be bounded below on weighted Bergman spaces. We state the result here, but the key is Proposition \ref{luk}, proved afterward.

\begin{thm} 
\label{tgap}
Let $1 \leq p < \infty$, $\alpha > -1$. $T_g$ is bounded below on $A_{\alpha}^p$ if and only if there exist $c > 0$ and $\delta > 0$ such that
$$ |  \{z \in D : |g'(z)|(1-|z|^2) > c \} \cap S(I) | > \delta |I|^2.  $$
\end{thm}

\begin{proof} We must assume $T_g$ is bounded on $A^p_{\alpha}$. By Theorem \ref{nec}, $g \in \B$. (That this is also sufficient for $T_g$ to be bounded on $A^p_0$ is in \cite{AC}.) $T_g$ is bounded below on $A^p_{\alpha}$ if and only if
$$ \|  T_gf  \|_{A^p_{\alpha}}^p \approx \int_D |f(z)|^p|g'(z)|^p (1-|z|^2)^{\alpha + p} \, dA(z) \gtrsim \|  f  \|^p_{A^p_{\alpha}}. $$
By Proposition \ref{luk}, this is true if and only if there exist $c > 0$ and $\delta > 0$ such that
$$ | \{z \in D : |g'(z)|^p(1-|z|^2)^{ p}  > c \} \cap S(I) | > \delta |I|^2 $$
for all arcs $I \subseteq \partial D$. If this holds for some $p$ it holds for all $p$.
\end{proof}

The proof of \cite[Proposition 5.4]{RaU} shows this result is nonvacuous. Ramey and Ullrich construct a Bloch function $g$ such that $|g'(z)|(1-|z|) > c_0$ if $1 - q^{-(k+1/2)} \leq |z| \leq 1 - q^{-(k+1)}$, for some $c_0 > 0$, $q$ some large positive integer, and $k = 1, 2, \dots$. Given a Carleson square $S(I)$, let $k_I$ be the least positive integer such that $q^{-k_I + 1/2} \leq |I|$. The annulus $E = \{z : 1 - q^{-(k_I+1/2)} \leq |z| \leq 1 - q^{-(k_I+1)}\}$ intersects $S(I)$, and 
$$ |E \cap S(I)| \approx |I|((1 - q^{-(k_I+1)}) - (1 - q^{-(k_I+1/2)})) = |I| \frac{q^{1/2}-1}{q^{k_I+1}} \geq \frac{q^{1/2}-1}{q^{3/2}} |I|^2. $$
Setting $c = c_0$ and $\delta \approx 1/q$ show Theorem \ref{tgap} holds for this example of $g$, and $T_g$ is bounded below on $A^p_\alpha$.

We define $H^p_0 = H^p/\C = \{f \in H^p : f(0) = 0\}$. The operator $S_g$ can clearly be bounded below, since $g(z) = 1$ gives the identity operator. A result due to Luecking (see \cite[3.34]{CM}) leads to a characterization of functions for which $S_g$ is bounded below on $H^2_0$ and $A^p_\alpha/\C$. We state a reformulation useful to our purposes here.

\begin{prop} \emph{(Luecking)}
\label{luk}
Let $\tau$ be a bounded, nonnegative, measurable function on $D$. Let $G_c = \{ z \in D : \tau(z) > c \}$, $1 \leq p < \infty$, and $\alpha > -1$. There exists $C > 0$ such that the inequality 
$$ \int_D |f(z)|^p\tau(z) (1-|z|)^{\alpha}\, dA(z) \geq C \int_D |f(z)|^p (1-|z|)^{\alpha} \, dA(z) $$
holds if and only if there exist $\delta > 0$ and $c > 0$ such that $|G_c \cap S(I)| \geq \delta |I|^2$ for every interval $I \subset \mathbb{T}$.
\end{prop}

The proof is omitted. Using the Littlewood-Paley identity we get the following:
\begin{cry}
\label{lukcry}
 $S_g$ is bounded below on $H^2_0$ if and only if there exist $c > 0$ and $\delta > 0$ such that $|G_c \cap S(I)| \geq \delta |I|^2$, where $G_c = \{z \in D: |g(z)| > c \}$.
\end{cry}

We use Corollary \ref{lukcry} to construct a nonexample of boundedness below of $S_g$ on $H^2_0$, and compare $M_g$ on $H^2$ to $S_g$ on $H^2_0$. If $g(z)$ is the singular inner function $\exp( \frac{z+1}{z-1})$, $S_g$ is not bounded below on $H^2_0$. To see this, fix $c \in (0,1)$. $G_c$ is the complement in $D$ of a horodisk, a disk tangent to the unit circle, with radius $r = \frac{\log c + 1}{2(\log c - 1)}$ and center $1-r$. Choosing a sequence of intervals $I_n \subset \mathbb{T}$ such that 1 is the center of $I_n$ and $|I_n| \to 0$ as $n \to \infty$, we see
$$ \frac{|G_c \cap S(I_n)|}{|I_n|^2} \to 0 \text{ as } n \to \infty, $$
meaning $S_g$ is not bounded below on $H^2_0$.

$M_g$ is bounded below on $H^2$ if and only if the radial limit function of $g \in H^\infty$ is essentially bounded away from $0$ on $\partial D$. (\cite{KP} has this result as a special case of weighted composition operators.) Theorem \ref{mg_sg} will show this is weaker than the condition for $S_g$ to be bounded below on $H^2_0$. The example above of a singular inner function then shows it is strictly weaker. To prove Theorem \ref{mg_sg} we use a lemma which allows us to estimate an analytic function inside the disk by its values on the boundary. Define the conelike region with aperture $\alpha \in (0, 1)$ at $e^{i\theta}$ to be
$$ \Gamma_\alpha(e^{i\theta}) = \left\{ z \in D : \frac{|e^{i\theta} - z|}{1 - |z|} < \alpha \right\}. $$
For a function $g \in H(D)$, define the nontangential limit function, for almost all $e^{it}$,
$$ |g^*(e^{it})| = \lim_{\Gamma_\alpha(e^{it}) \ni z \to e^{it}} |g(z)|. $$
For any arc $I \subseteq \partial D$ and $0 < r < 2\pi/|I|$, $rI$ will denote the arc with the same center as $I$ and length $r|I|$. We define the upper Carleson rectangle $S_\varepsilon(I) = \{re^{it} : 1 - |I| < r < (1 - \varepsilon|I|), e^{it} \in I\}$, and $S_+(I) = S_{1/2}(I)$.

\begin{lemma} 
\label{star}
Given $(1 >)\varepsilon > 0$ and a point $e^{i\theta}$ such that $|g^*(e^{i\theta})| < \varepsilon$, there exists an arc $I \subset \partial D$ such that $|g(z)| < \varepsilon$ for $z \in S_{\varepsilon}(I)$.
\end{lemma}
\begin{proof}
We can choose $\alpha$ close enough to $1$ so that $S_\varepsilon(I) \subset \Gamma_\alpha(e^{i\theta})$ for all $I$ centered at $e^{i\theta}$ with, say, $|I| < 1/4$. If $|g^*(e^{i\theta})| < \varepsilon$, there exists $\delta > 0$ such that $z \in \Gamma_\alpha(e^{i\theta}), |z - e^{i\theta}| < \delta$ imply $|g(z)| < \varepsilon$. Choosing $I$ such that $S(I)$ is contained in a $\delta$-neighborhood of $e^{i\theta}$ finishes the proof.
\end{proof}

\begin{thm}
\label{mg_sg}
If $S_g$ is bounded below on $H^2_0$, then $M_g$ is bounded below on $H^2$.
\end{thm}
\begin{proof}
Assume $M_g$ is not bounded below on $H^2$. Let $\varepsilon > 0$. The radial limit function of $g$ equals $g^*$ almost everywhere, so there exists a point $e^{i\theta}$ such that $|g^*(e^{i\theta})| < \varepsilon$. By Lemma \ref{star}, there exists $S(I)$ such that $|\{z : |g(z)| \geq \varepsilon \} \cap S(I)| \leq \varepsilon|I|$. Since $\varepsilon$ was arbitrary, this violates the condition in Proposition \ref{luk}.
\end{proof}

We now characterize the symbols $g$ which make $S_g$ bounded below on the Bloch space. It turns out to be a common condition appearing in a few different forms in the literature. The condition appears in characterizing $M_g$ on $A^2_0$ in McDonald and Sundberg \cite{McSu}. Our main result is equivalence of (i)-(iii) in Theorem \ref{main}, and we give references with brief explanations for (iv)-(vi).

\begin{thm} 
\label{main}
The following are equivalent for $g \in H^{\infty}$:
\begin{itemize}
\item[(i)] $g = BF$ for a finite product $B$ of interpolating Blaschke products and $F$ such that $F, 1/F \in H^{\infty}$.
\item[(ii)] $S_g$ is bounded below on $\B/\C$.
\item[(iii)] There exist $r < 1$ and $\eta > 0$ such that for all $a \in D$, 
$$ \sup_{z \in D(a,r)} |g(z)| > \eta.$$
\item[(iv)] $S_g$ is bounded below on $H^2_0$.
\item[(v)] $M_g$ is bounded below on $A^p_{\alpha}$ for $\alpha > -1$.
\item[(vi)] $S_g$ is bounded below on $A^p_\alpha/\C$ for $\alpha > -1$.
\end{itemize}
\end{thm}

\begin{proof}
(i) $\Rightarrow$ (ii): Note that $S_{g_1g_2} = S_{g_1}S_{g_2}$ for any $g_1, g_2$. It follows that if $S_{g_1}$ and $S_{g_2}$ are bounded below then $S_{g_1g_2}$ is also bounded below. We will show that $S_F$ and $S_B$ are bounded below, implying the result for $S_g$.

It is necessary that $g \in H^\infty$ for $S_g$ to be bounded on $\B$. (Corollary \ref{cor_bdd}) If $F, 1/F \in H^{\infty}$, then 
$$ \| S_Ff \| = \sup_{z \in D} |F(z)||f'(z)|(1-|z|^2) \geq (1/\|1/F\|_\infty) \|f\|_\B. $$
Hence $S_F$ is bounded below.

By virtue of the fact beginning this proof, we may assume $B$ is a single interpolating Blaschke product without loss of generality. Let $\{w_n\}$ be the zero sequence of $B$, so
$$ B(z) = e^{i\phi} \prod_n \frac{w_n - z}{1 - \bar{w}_n z}. $$
Denote the pseudohyperbolic metric 
$$ \rho(z, w) = \frac{|w - z|}{|1 - \bar{w} z|}, \text{ for any } z, w \in D. $$
For the pseudohyperbolic disk of radius $d > 0$ and center $w \in D$, we use the notation
$$ D(w, d) = \{ z \in D : \rho(z, w) < d \}. $$
Let $B_j$ be $B$ without its $j$th zero, i.e., $B_j(z) = \frac{1 - \bar{w}_j z}{w_j - z} B(z)$. Since $B$ is interpolating, there exist $\delta > 0$ and $r > 0$ such that, for all $j$, $|B_j(z)| > \delta$ whenever $z \in D(w_j, r)$. In particular, the sequence $\{w_n\}$ is separated, so shrinking $r$ if necessary, we may assume
$$ \inf_{j \neq k} \rho(w_k, w_j) > 2r. $$
We compare $\|f\|$ to $\|S_Bf\| = \sup_{z \in D} |B(z)||f'(z)|(1-|z|^2)$. Let $a \in D$ be a point where the supremum defining the norm of $f$ is almost achieved, say, $|f'(a)|(1-|a|^2) > \|f\|/2$.

Consider the pseudohyperbolic disk $D(a, r)$. Inside $D(a, r)$ there may be at most one zero of $B$, say $w_k$. We examine three cases depending on the location and existence of $w_k$.

If $r/2 \leq \rho(w_k, a) < r$, then
$$ |B(a)| = \frac{|w_k - a|}{|1 - \bar{w}_k a|} |B_k(a)| > (r/2)\delta. $$
Thus we would have 
$$ \|S_Bf\| \geq |B(a)||f'(a)|(1-|a|^2) > (r/2)\delta \|f\|/2, $$
and $S_g$ would be bounded below.

On the other hand, suppose $\rho(w_k, a) < r/2$. Consider the disk $D(w_k, r/2)$, which is contained in $D(a, r)$. The expression $1-|z|^2$ is roughly constant on a pseudohyperbolic disk, i.e.,
$$ \sup_{z \in D(a,r)} (1-|z|^2) > C_r(1-|a|^2) \text{ for some } C_r > 0. $$
$C_r$ does not depend on $a$, and is near $1$ for small $r$. By the maximum principle for $f'$, there exists a point $z_a \in \partial D(w_k, r/2)$ where
$$ |f'(z_a)|(1-|z_a|^2) > |f'(a)|C_r(1-|a|^2) > C_r \|f\|/2. $$
(Since $\rho(w_k, a) < r/2$ and $\rho(z_a, w_k) = r/2$, we have $\rho(z_a, a) < r$.) This shows that $S_g$ is bounded below, for
$$ \|S_Bf\| \geq |B(z_a)||f'(z_a)|(1-|z_a|^2) $$
$$ > \rho(w_k, z_a) |B_k(z_a)| C_r \|f\|/2 $$
$$ > (r/2) \delta C_r \|f\|/2. $$
Finally, suppose no such $w_k$ exists. Then the function $(\frac{a-z}{1-\bar{a}z})B(z)$ is also an interpolating Blaschke product, and the previous case applies with $w_k = a$.

(ii) $\Rightarrow$ (iii): Assume (iii) fails. Given $\varepsilon > 0$, choose $r$ near $1$ so that $1 - r^2 < \varepsilon$, and choose $a \in D$ such that $|g(z)| < \varepsilon$ for all $z \in D(a, r)$. Consider the test function $f_a(z) = (a - z)/(1 - \bar{a} z)$. By a well-known identity,
$$ (1-|z|^2)|f'_a(z)| = 1 - (\rho(a, z))^2. $$
Thus $f_a \in \B$ with $\|f_a\| = 1$ for all $a \in D$. (The seminorm is 1, but the true norm is between 1 and 2 for all $a$.) By supposition on $g$,
$$ \|S_g f_a\| = \sup_{z \in D} |g(z)||f'_a(z)|(1-|z|^2) $$
$$ = \max \left\{ \sup_{z \in D(a,r)} |g(z)||f'_a(z)|(1-|z|^2), \sup_{z \in D \setminus D(a,r)} |g(z)||f'_a(z)|(1-|z|^2) \right\} $$
$$ \leq \max \left\{ \sup_{z \in D(a,r)} |g(z)|\|f_a\|, \sup_{z \in D \setminus D(a,r)} |g(z)|(1-r^2) \right\} $$
$$ < \max \{ \varepsilon, \|g\|_\infty \varepsilon \} \leq \varepsilon(\|g\|_\infty + 1) $$
Since $\|f_a\| = 1$ and $\varepsilon$ was arbitrary, $S_g$ is not bounded below.

(iii) $\Rightarrow$ (i): Assuming (iii) holds, we first rule out the possibility that $g$ has a singular inner factor. We factor $g = B I_g O_g$ where $B$ is a Blaschke product, $I_g$ a singular inner function, and $O_g$ an outer function. Let $\nu$ be the measure on $\partial D$ determining $I_g$, so
$$ I_g(z) = \exp \left( - \int \frac{e^{i\theta} + z}{e^{i\theta} - z} d\nu(\theta) \right). $$
Let $\varepsilon > 0$. For any $\alpha > 1$ and for $\nu$-almost all $\theta$, there exists $\delta > 0$ such that
$$ z \in \Gamma_\alpha(e^{i\theta}), |z - e^{i\theta}| < \delta \text{ imply } |I_g(z)| < \varepsilon. \quad (3.1) $$
This is \cite[Theorem II.6.2]{Gar}. $\delta$ may depend on $\theta$ and $\alpha$, but for nontrivial $\nu$ there exists some $\theta$ where $(3.1)$ holds. Given $r < 1$, choose $\alpha < 1$ such that, for every $a$ near $e^{i\theta}$ on the ray from $0$ to $e^{i\theta}$, the pseudohyperbolic disk $D(a, r)$ is contained in $\Gamma_\alpha(e^{i\theta})$. The disk $D(a, r)$ is a euclidean disk whose euclidean radius is comparable to $1 - |a|$. For $a$ close enough to $e^{i\theta}$,
$$ z \in D(a, r) \text{ implies } |z - e^{i\theta}| < \delta. $$
Hence $\sup_{z \in D(a, r)} |g(z)| < \varepsilon \|g\|$. This violates (iii), so $\nu$ must be trivial, and $I_g \equiv 1$.

A similar argument handles the outer function $O_g$. If for all $\varepsilon > 0$ there exists $e^{it}$ such that $|O^*_g(e^{it})| < \varepsilon$, we apply Lemma \ref{star}. The upper Carleson square in Lemma \ref{star} contains some pseudohyperbolic disk that violates (iii), so $O^*_g$ is essentially bounded away from $0$. There exists $\eta > 0$ such that $|O^*_g(e^{it})| \geq \eta$ almost everywhere. Note $1/O_g \in H^\infty$, since for all $z \in D$,
$$ \log|O_g(z)| = \frac{1}{2\pi} \int_0^{2\pi} \log|O^*_g(e^{it})| \frac{1-|z|^2}{|e^{it}-z|^2} dt \geq \log \eta. $$

We have reduced the symbol to a function $g = BF$, where $F, 1/F \in H^\infty$ and $B$ is a Blaschke product, say with zero sequence $\{w_n\}$. We will show that the measure $\mu_B = \sum (1 - |w_n|^2) \delta_{w_n}$ is a Carleson measure, implying $B$ is a finite product of interpolating Blaschke products. (see, e.g., \cite[Lemma 21]{McSu}) Let $r < 1$ and $\eta > 0$ be as in (iii), so $\sup_{z \in D(a,r)} |B(z)| > \eta$ for all $a$. Given any arc $I \subseteq \partial D$, we may choose $a_I$ and $z_I$ such that $D(a_I, r) \subseteq S(I)$, $z_I \in D(a_I, r)$, $|B(z_I)| > \eta$, and $(1 - |z_I|) \approx |I|$ as $I$ varies. $\mu_B(S(I)) = \sum (1 - |w_{n_k}|^2)$ where the subsequence $\{w_{n_k}\} = \{w_n\} \cap S(I)$. Assume without loss of generality that $|I| < 1/2$, so $|w_{n_k}| > 1/2$ for all $k$. This ensures $|1 - \bar{w}_{n_k} z_I| \approx |I|$. Thus we have
$$ \frac{1}{|I|} \sum_k (1 - |w_{n_k}|^2) \approx \sum_k \frac{(1-|z_I|^2)(1-|w_{n_k}|^2)}{|1 - \bar{w}_{n_k} z_I|^2} = \sum_k 1 - (\rho(z_I, w_{n_k}))^2 $$
$$ < 2 \sum_n 1 - \rho(z_I, w_n) \leq - \sum_n \log \rho(z_I, w_n) = - \log \prod_n \frac{|w_n - z_I|}{|1 - \bar{w}_n z_I|} = - \log |B(z_I)| \leq - \log \eta. $$
This shows $\mu_B$ is a Carleson measure.

(i) $\Leftrightarrow$ (iv) $\Leftrightarrow$ (v) $\Leftrightarrow$ (vi)
Bourdon shows in \cite[Theorem 2.3, Corollary 2.5]{B} that (i) is equivalent to the reverse Carleson condition in Corollary \ref{lukcry} above, hence (i) $\Leftrightarrow$ (iv). This reverse Carleson condition also characterizes boundedness below of $M_g$ on weighted Bergman spaces by Proposition \ref{luk}. Thus (iv) $\Leftrightarrow$ (v). A key connection is between $S_g$ and $M_g$ via the differentiation operator and equation (\ref{isom}), since $(S_gf)' = M_gf'$. The following diagram is commutative:
$$ \begin{CD}
A^p_\alpha/\C @>S_g>> A^p_\alpha/\C\\
@VVf \mapsto f'V @VVf \mapsto f'V\\
A^p_{\alpha+p} @>M_g>> A^p_{\alpha+p}
\end{CD} $$
This explains (v) $\Rightarrow$ (vi). Since $A^2_{-1} = H^2$, we can combine (iv) and (vi) to say $S_g$ is bounded below on $A^2_\alpha/\C$ for $\alpha \geq -1$.
\end{proof}

\subsection*{Concluding Remarks}
We suspect the results about $H^2$ can be extended to all $H^p$, $1 \leq p < \infty$, but without the Littlewood-Paley identity the proof is more difficult. Generalizing the results on Bloch to the $\alpha$-Bloch spaces can be done with adjusted test functions as in \cite{ZgCh}. Finally, we have partial results concerning $S_g$ being bounded below on BMOA, but have not completed proving a characterization like the one in Theorem \ref{main}.

I would like to thank my advisor Dr. Wayne Smith for his invaluable guidance.

\noindent Department of Mathematics, University of Hawaii, Honolulu, Hawaii 96822\\
E-mail address: austina@hawaii.edu

\end{document}